\documentclass[11pt,a4paper]{article}
\usepackage{pdflscape}
\usepackage{indentfirst,mathrsfs}
\usepackage{amsfonts,amsmath,amssymb,amsthm}
\usepackage{latexsym,amscd}
\usepackage{amsbsy}
\usepackage{color}
\usepackage{CJK}

\newtheorem{theorem}{Theorem}

\begin{document}
 \title{On the maximum order complexity of the  Thue-Morse and Rudin-Shapiro sequence}

 \author{Zhimin Sun$^1$ and Arne Winterhof$^2$\\ $^1$   Faculty of Mathematics and Statistics, Hubei Key Laboratory of Applied \\ Mathematics, Hubei
University, Wuhan, 430062,
 China\\
 $^2$ Johann Radon Institute for Computational and Applied
Mathematics,\\
Altenberger Stra{\ss}e 69, A-4040 Linz, Austria\\
e-mail: arne.winterhof@oeaww.ac.at}

\maketitle

\begin{abstract}
 Expansion complexity and maximum order complexity are both finer measures of pseudorandomness than the linear complexity which is the most prominent quality measure for cryptographic sequences.
 The expected value of the $N$th maximum order complexity is of order of magnitude $\log N$  whereas it is easy to find families of sequences with
 $N$th expansion complexity exponential in $\log N$.
 This might lead to the conjecture that the maximum order complexity is a finer measure than the expansion complexity. However, in this paper we provide two examples, the Thue-Morse sequence
 and the Rudin-Shapiro sequence with very small expansion complexity but very large maximum order complexity. More precisely, we prove explicit formulas for their $N$th
 maximum order complexity which are both of
  largest possible order of magnitude $N$.  We present the result on the Rudin-Shapiro sequence in a more general form as a formula for the maximum order complexity
  of certain pattern sequences.
\end{abstract}


 \section{Introduction}

  \subsection{Motivation}

 For a sequence ${\cal S}=(s_i)_{i=0}^\infty$ over the finite field ${\mathbb F}_2$ of two elements and a positive integer $N$, the {\em $N$th linear complexity} $L({\cal S},N)$
is the length $L$ of a shortest linear recurrence
$$
 s_{i+L}=\sum_{\ell=0}^{L-1} c_\ell s_{i+\ell}, \quad 0\le i\le N-L-1,
$$
with coefficients $c_\ell\in {\mathbb F}_2$, which is satisfied by the first $N$ terms of the sequence.

The ($N$th) linear complexity is a measure for the unpredictability of a sequence and thus its suitability in cryptography. A sequence ${\cal S}$ with small $L({\cal S},N)$ for a sufficiently large $N$
is disastrous for cryptographic applications. However, the converse is not true. There are highly predictable sequences ${\cal S}$ with large $L({\cal S},N)$, including the example
\begin{equation}\label{trivex} s_0=\ldots=s_{N-2}=0\ne s_{N-1}.
\end{equation}
Hence, for testing the suitability of a sequence in cryptography we also have to study finer figures of merit. A recent survey on
linear complexity and related measures is given in~\cite{MW13}.

The \emph{$N$th maximum order complexity} $M({\cal S},N)$ (or {\em $N$th nonlinear complexity}) of a binary sequence ${\cal S}=(s_i)_{i=0}^\infty$
with $(s_0,\ldots,s_{N-2})\ne (a,\ldots,a)$ and $a\in \{0,1\}$ is the smallest positive integer $M$
such that there is a polynomial $f(x_1,\ldots,x_M)\in {\mathbb F}_2[x_1,\ldots,x_M]$
with
$$s_{i+M}=f(s_i,s_{i+1},\ldots,s_{i+M-1}),\quad 0\le i\le N-M-1,$$
see \cite{ja89,ja91,nixi14}. If $s_i=a$ for $i=0,\ldots,N-2$, we define $M({\cal S},N)=0$ if $s_{N-1}=a$ and $M({\cal S},N)=N-1$ if $s_{N-1}\ne a$.

Obviously we have
$$M({\cal S},N)\le L({\cal S},N).$$
We have $M({\cal S},N)=L({\cal S},N)-1$ for the example $(\ref{trivex})$.
However, the expected value of $M({\cal S},N)$ is of order of magnitude $\log N$, see \cite{ja89} and also  \cite{ermu,jabo,nixi14}, and the expected value of $L(N)$ is $N/2+O(1)$ by \cite{gu}.
Hence, the maximum order complexity is a finer measure of pseudorandomness than the linear complexity.

Diem \cite{di12} introduced the expansion complexity of the sequence ${\cal S}$ as follows. We define the {\em generating function} $G(x)$ of ${\cal S}$ by
$$G(x)=\sum_{i=0}^\infty s_i x^i,$$
viewed as a formal power series over ${\mathbb F}_2$. (Note the change by the factor $x$ compared to the definition in~\cite{di12}.)
For a positive integer $N$, the {\em $N$th expansion complexity} $E_N=E_N({\cal S})$ is $E_N=0$ if $s_0=\ldots=s_{N-1}=0$ and otherwise the least total degree
of a nonzero polynomial $h(x,y)\in {\mathbb F}_2[x,y]$ with
$$h(x,G(x))\equiv 0 \bmod x^N.$$
 By \cite[Theorem 3]{meniwi} we have
 $$E({\cal S},N)\le L({\cal S},N)+1$$
  and also in \cite{meniwi} examples of sequences ${\cal S}$ are given with $E({\cal S},N)$ substantially smaller than $L({\cal S},N)$.
 Hence, the expansion complexity is also a finer measure of pseudorandomness than the linear complexity.
 In particular, for (ultimately) non-periodic automatic sequences we have seen in  \cite{mewi} that
 they have bounded expansion complexity but linear complexity of order of magnitude $N$.

 Now it is a natural question to compare the two finer measures of pseudorandomness, expansion complexity and maximum order complexity.
On the one hand, by \cite[Theorem~1]{meniwi} for any $T$-periodic sequence ${\cal S}$ and $N>T(T-1)$ we have $E({\cal S},N)=L({\cal S},N)+1$ which has an expected value of order of magnitude $T$,
see for example \cite{MW13}. On the other hand, the expected value of $M({\cal S},N)$ is of order of magnitude $\log N$.
This might lead to the conjecture that $M({\cal S},N)$ is a finer measure of pseudorandomness than $E({\cal S},N)$.
However, in this paper we will disprove this conjecture by showing that certain pattern sequences which include the Thue-Morse and the Rudin-Shapiro sequence have bounded expansion complexity
but maximum order
complexity of largest possible order of magnitude $N$.
We explain this more precisely in the next subsection.

\subsection{Results of this paper}

  The {\em Thue-Morse sequence} ${\cal T}=(t_i)_{i=0}^\infty$ over ${\mathbb F}_2$ is defined by
\begin{equation}\label{tmdef} t_i=\left\{ \begin{array}{cl} t_{i/2} & \mbox{if $i$ is even},\\ t_{(i-1)/2}+1 & \mbox{if $i$ is odd},
              \end{array}\right.\quad i=1,2,\ldots
\end{equation}
with initial value $t_0=0$. In other words $t_i$ is the parity of the sum of digits of $i$. Taking
$$h(x,y)=(x+1)^3 y^2+(x+1)^2 y+x$$
its generating function $G(x)$ satisfies $h(x,G(x))=0$
and thus
$$E({\cal T},N)\le 5,\quad N=1,2,\ldots$$
Theorem~\ref{thuethm} below gives an explicit formula for $M({\cal T},N)$ of order of magnitude $N$.

 More generally, for a positive integer $k$ we study the {\em pattern sequence} ${\cal P}_{k}=(p_i)_{i=0}^\infty$ over ${\mathbb F}_2$ defined by
\begin{equation}\label{pattdef} p_i=\left\{ \begin{array}{cl}
 p_{\lfloor i/2 \rfloor}+1 & \mbox{if $i\equiv -1 \,\, \bmod \,\, 2^k$},\\
p_{\lfloor i/2 \rfloor} & \mbox{otherwise},
\end{array}\right.\quad i=1,2,\ldots
\end{equation}
with initial value $p_0=0$.
In other words $p_i$ is the parity of the number of occurences of the all one pattern of length $k$ in the binary expansion of $i$.
For $k=1$ we get the Thue-Morse sequence and for $k=2$ the {\em Rudin-Shapiro sequence}.

Taking
$$h(x,y)=(x+1)^{2^{k}+1}y^2+(x+1)^{2^k}y+x^{2^k-1}$$
its generating function $G(x)$ satisfies $h(x,G(x))=0$
and thus
$$E({\cal P}_k,N)\le 2^k+3,\quad N=1,2,\ldots$$
Theorem~\ref{patternthm} below  provides an explicit formula for $M({\cal P}_k,N)$ for $k\ge 2$ of order of magnitude $N$.
Note that the case $k=1$ is slightly different than the case $k\ge 2$.

In Section~\ref{thue} we study the maximum order complexity of the Thue-Morse sequence, that is, ${\cal P}_1$ and
in Section~\ref{rudin} of ${\cal P}_k$ for $k\ge 2$.

 \section{Thue-Morse sequence}
\label{thue}

 \begin{theorem}\label{thuethm}
  For $N\ge 4$,
  the $N$th maximum order complexity of the Thue-Morse sequence ${\cal T}$ satisfies
  $$M({\cal T},N)=2^\ell+1,$$
  where
  $$\ell=\left\lceil \frac{\log (N/5)}{\log 2}\right\rceil.$$
 \end{theorem}
\begin{proof}
 For $N=4,5,6$ the result is easy to verify.

 By the monotony of the maximum order complexity it is enough to show
 $$M({\cal T},5\cdot 2^{\ell-1}+1)\ge 2^\ell +1\ge M({\cal T},5\cdot 2^\ell)\quad \mbox{for }\ell=1,2,\ldots$$
 From Proposition 3.1 in \cite{ja89}, if $t$ be the length of the longest subsequence of ${\cal T}$ that
occurs at least twice with different successors, then ${\cal T}$  has
 the maximum order complexity $t + 1$. Hence
 the first inequality follows from
 \begin{equation}\label{ti} t_i=t_{i+3\cdot2^{\ell-1}} \quad \mbox{for }i=0,1,\ldots,2^\ell -1 \quad \mbox{and}\quad t_{2^\ell}\ne t_{5\cdot 2^{\ell-1}},\quad \ell=1,2,\ldots
 \end{equation}
 which we show by induction over $\ell$ below. More precisely, if there was a recurrence of length $2^\ell$ for the first $5\cdot 2^{\ell-1}+1$ sequence elements,
 $$t_{i+2^\ell}=f(t_i,\ldots,t_{i+2^{\ell-1}}),\quad 0\le i\le 3\cdot 2^{\ell-1},$$
 then from $(t_0,\ldots,t_{2^\ell-1})=(t_{3\cdot 2^{\ell-1}},\ldots,t_{5\cdot 2^{\ell-1}-1})$ we would get $t_{2^\ell}=t_{5\cdot 2^{\ell-1}}$, a contradiction to $(\ref{ti})$.

 For $\ell=1$ the assertion $(\ref{ti})$ is obviously true and we may assume $\ell\ge 2$.\\
 For even $i$ we get by $(\ref{tmdef})$ and induction
 $$t_i=t_{i/2}=t_{i/2+3\cdot 2^{\ell -2}}=t_{i+3\cdot 2^{\ell-1}},\quad i=0,2,\ldots,2^\ell-2.$$
 For odd $i$ we get
 $$t_i=t_{(i-1)/2}+1=t_{(i-1)/2+3\cdot 2^{\ell-2}}+1=t_{i+3\cdot 2^{\ell-1}},\quad i=1,3,\ldots,2^\ell-1.$$
 Moreover,
 $$t_{2^\ell}=t_{2^{\ell -1}}\ne t_{5\cdot 2^{\ell-2}}=t_{5\cdot 2^{\ell-1}}.$$

Now we prove $M({\cal T},5\cdot 2^\ell)\le 2^\ell+1$ for $\ell=1,2,\ldots$
In other words, we have to show that for any $\ell=1,2,\ldots$,
if for some $0\le j<k\le  2^{\ell+2}-2$ we have
\begin{equation}\label{pre} t_{i+j}=t_{i+k} \quad \mbox{for }i=0,1,\ldots,2^\ell,
\end{equation}
then we also have  $t_{2^\ell+1+j}= t_{2^\ell+1+k}$.
This can be easily verified for $\ell=1$ and we may assume $\ell\ge 2$.

First we note that $(t_j,t_{j+1},t_{j+2},t_{j+3})$ is of the form $(x,x+1,y,y+1)$ if $j$ is even since
$t_{2m+1}=t_{m}+1=t_{2m}+1$ and either of the form $(x,x,x+1,y)$  for $j\equiv 1\,\, \bmod \,\, 4$
or $(x,y,y+1,y+1)$ for $j\equiv 3\,\, \bmod \,\, 4$ since
$t_{4m+1}=t_m+1=t_{4m+2}$ and $t_{4m+3}=t_m=t_{4m}$.
 Hence,  $(t_j,t_{j+1},t_{j+2},t_{j+3})=(t_k,t_{k+1},t_{k+2},t_{k+3})$ implies $j\equiv k\bmod 2$.

If $j$ and $k$ are both even,
then from $(\ref{tmdef})$ and $(\ref{pre})$ with $i=2^\ell$ we get
$$t_{2^\ell+1+j}=t_{2^{\ell-1}+j/2}+1=t_{2^\ell+j}+1=t_{2^\ell+k}+1=t_{2^\ell+k+1}.$$

If $j$ and $k$ are both odd,
then $(\ref{pre})$ implies for any even $i$
$$t_{i/2+(j-1)/2}=t_{i+j}+1=t_{i+k}+1=t_{i/2+(k-1)/2} \quad \mbox{for }i=0,2,\ldots,2^\ell$$
and by induction
$$t_{2^\ell+1+j}=t_{2^{\ell-1}+(j+1)/2}= t_{2^{\ell-1}+(k+1)/2}=t_{2^\ell+1+k},$$
which completes the proof.
\end{proof}

 Remark 1. It is easy to see that $\frac{N}{5}+1\le M({\cal T},N)\le 2\frac{N-1}{5}+1$ for $N\ge 4$ and $M({\cal T},1)=0$, $M({\cal T},2)=M({\cal T},3)=1$.

\section{Pattern sequences}
\label{rudin}

\begin{theorem}\label{patternthm}
  For $k\ge 2$ and $N\geq 2^{k+3}-7$,
  the $N$th maximum order complexity of the pattern sequence ${\cal P}_k$ satisfies
  $$M({\cal P}_k,N)=(2^{k-1}-1) 2^\ell+1$$
  where
  $$\ell= \left\lceil  \frac{\log (N/(2^k-1))}{\log 2}\right\rceil -1.$$
 \end{theorem}

 \begin{proof}

By the monotony of the maximum order complexity it is enough to show
 $$M({\cal P}_k,(2^{k}-1)2^{\ell}+1)\ge (2^{k-1}-1)2^\ell +1\ge M({\cal P}_k,(2^{k}-1) 2^{\ell+1})\quad \mbox{for }\ell\ge 3.$$
From Proposition 3.1 in \cite{ja89}, the first inequality follows from
\begin{eqnarray}\label{patteq1}
p_i&=&p_{i+2^{\ell+k-1}} \quad \mbox{for}\quad i=0,1,\ldots,(2^{k-1}-1)2^\ell -1  \\
 &\mbox{and} &  p_{(2^{k-1}-1) 2^\ell}\ne p_{(2^{k}-1)2^\ell} \nonumber
\end{eqnarray}
for $\ell\ge  0$,
 which we show by induction over $\ell$.
For $\ell= 0$ the assertion is obviously true  since $p_i=0$ for $i=0,1, \ldots,2^k-2$ and $p_{2^k-1}=1$ by $(\ref{pattdef})$. We may assume $\ell\ge 1$.\\
  For even $i$ we get from $(\ref{pattdef})$ and induction
 \begin{equation}\label{patterneven}
 p_i=p_{i/2}=p_{i/2+2^{\ell+k -2}}=p_{i+2^{\ell+k -1}},\quad i=0,2,\ldots,(2^{k-1}-1)2^\ell -2.
 \end{equation}
For odd $i$ we get from $(\ref{pattdef})$
 \begin{equation}\label{pattodd}
      p_i=\left\{ \begin{array}{cl} p_{i-1} & \mbox{if $i\not\equiv -1 \bmod  2^k$ },\\
      p_{i-1}+1 & \mbox{if $i\equiv  -1 \bmod  2^k$},
 \end{array}\right.\quad i=1,3,\ldots
\end{equation}
 Now fix any odd $i=1,3,\ldots,(2^{k-1}-1)2^\ell -1$.
If $i\not\equiv -1 \bmod 2^k$, then  we get from (\ref{patterneven}) and (\ref{pattodd})
$$p_i=p_{i-1}=p_{i-1+2^{\ell+k-1}}=p_{i+2^{\ell+k-1}}.$$
If $i\equiv -1\bmod 2^k$, then
$$p_i=p_{i-1}+1=p_{i-1+2^{\ell+k-1}}+1=p_{i+2^{\ell+k-1}}.$$
 Moreover,
 $$p_{(2^{k-1}-1)2^\ell}=p_{(2^{k-1}-1)2^{\ell-1}} \ne p_{(2^{k}-1) 2^{\ell-1}}= p_{(2^{k}-1)2^\ell}$$
  by induction.

Now we prove $M({\cal P}_k,(2^{k}-1)2^{\ell+1})\leq (2^{k-1}-1)2^\ell+1 \,\, \mbox{for }\ell\ge 3.$

That is, we have to show for any $\ell\ge 3$ that,
if for some $0\le j<n\le
(3\cdot 2^{k-1}-1)2^\ell-2$
we have
\begin{equation}\label{sec3}
p_{i+j}=p_{i+n} \quad \mbox{for }i=0,1,\ldots,(2^{k-1}-1)2^\ell,
\end{equation}
then we also have
\begin{equation}\label{tobeshown} p_{(2^{k-1}-1)2^\ell+1+j}= p_{(2^{k-1}-1)2^\ell+1+n}.
\end{equation}

First we observe that $(\ref{sec3})$ implies $j\equiv n\bmod 2^k$:

We choose any $m_1,m_2$ with $n+m_1\equiv 2^k-1\bmod 2^{k+1}$ and $n+m_2\equiv -1 \bmod 2^{k+1}$ and see that
$$n+m_1\equiv n+m_2\equiv -1 \bmod 2^k,\quad (n+m_1-1)/2\equiv  2^{k-1}-1\bmod 2^k$$
and
$$(n+m_2-1)/2\equiv -1\bmod 2^k.$$

If $j\equiv n\bmod 2$, then $j+m_1\equiv 1\bmod 2$. Moreover, we assume $1\le m_1\le 2^{k+1}$ in this case. Now $(\ref{sec3})$ with $i\in \{m_1,m_1-1\}$ and $(\ref{pattodd})$ imply
$p_{j+m_1}=p_{n+m_1}=p_{n+m_1-1}+1=p_{j+m_1-1}+1$ and from $(\ref{pattodd})$  again we get $j+m_1\equiv -1\bmod 2^k$ and thus
$j\equiv n\bmod 2^k$ in this case.

If $j\not\equiv n\bmod 2$, we assume $2\le m_1,m_2\le 2^{k+1}+1$. Then from $(\ref{sec3})$ with $i\in \{m_1-1,m_1-2\}$, $(\ref{pattdef})$ and $(\ref{pattodd})$ we get
$p_{j+m_1-1}=p_{n+m_1-1}=p_{(n+m_1-1)/2}=p_{(n+m_1-3)/2}=p_{n+m_1-3}=p_{n+m_1-2}=p_{j+m_1-2}$ implies $j+m_1-1\not\equiv -1 \bmod 2^k$.
However,
 $p_{j+m_2-1}=p_{n+m_2-1}=p_{n+m_2-2}+1=p_{j+m_2-2}+1$ and $(\ref{pattodd})$ imply $j+m_2-1\equiv -1 \bmod 2^k$
in contradiction to $m_1\equiv m_2\bmod 2^k$.

It remains to show that  $(\ref{sec3})$ implies $(\ref{tobeshown})$ for any $j\equiv n\bmod 2^k$.

For $j\equiv n\equiv 0\bmod 2$, $(\ref{pattodd})$ and $(\ref{sec3})$ with $i=(2^{k-1}-1)2^\ell$ immediately imply $(\ref{tobeshown})$.
For $j\equiv n\equiv 1\bmod 2$ we prove the assertion by induction.

Note that from (\ref{patteq1}) we get the last $(2^{k-1}-1)2^{\ell+1}$ elements from the first ones:
\begin{eqnarray*}
p_{i+2^{\ell+k}}=p_i \quad \mbox{for } i=0,1,\ldots, (2^{k-1}-1)2^{\ell+1} -1.
\end{eqnarray*}
 Then for verifying our assertion for $\ell=3$ we need only the first $3\cdot2^{k+2}-7$
  elements of $P_k$.
We use the abbreviation $a^t=\underbrace{aa\ldots a}_{t}$ for the word of $t$ consecutive~$a$ and get
  using $(\ref{pattdef})$:
  \begin{equation*}
 {\cal P}_2= \big{(} 0^{3}10^{2}10^{4}1^3010^{3}10^{2}101^30^310^410^2100\ldots\big{)}
 \end{equation*}
 \begin{equation*}
{\cal P}_3= \big{(}0^{7}10^{6}10^{8}10^{4}1^2010^{7}1
 0^{6}10^{8}1^50^21
 0^{8}10^{6}10^{8}10\ldots\big{)}
\end{equation*}
  and for $k \geq 4$
 \begin{eqnarray*}
 {\cal P}_{k}&=& \big{(} 0^{2^k-1}10^{2^k-2}10^{2^k}10^{2^k-4}1^2010^{2^k-1}1
 0^{2^k-2}10^{2^k}10^{2^k-8}1^40^21 \\
 && \quad 0^{2^k}10^{2^k-2}10^{2^k}10^{2^k-7}\ldots\big{)}.
 \end{eqnarray*}
  Note that we have to compare only the patterns of length $(2^{k-1}-1)2^\ell+2$ starting with $p_j$ and $p_n$ with $j\equiv n\bmod 2^k$, $j\equiv n\equiv 1\bmod 2$ and $0\leq j<n\leq 2^{k+3}-1$.

Now we consider $\ell \ge 4$.
For even $i$ with $0\leq i\leq (2^{k-1}-1)2^\ell$ we get
from $(\ref{pattdef})$ and $(\ref{sec3})$
$$p_{i/2+(j-1)/2}=p_{i/2+(n-1)/2}.$$
From the observations above we know that this is only possible if $(j-1)/2\equiv (n-1)/2\bmod 2^k$.
Either by induction if $(j-1)/2\equiv (n-1)/2\equiv 1\bmod 2$ or using the already above verified result if $(j-1)/2\equiv (n-1)/2\equiv 0\bmod 2$, we get
\begin{eqnarray*}
p_{(2^{k-1}-1)2^\ell+1+j}
&=&p_{(2^{k-1}-1)2^{\ell-1}+(j+1)/2}=p_{(2^{k-1}-1)2^{\ell-1}+(n+1)/2}\\
&=&p_{(2^{k-1}-1)2^{\ell}+1+n},
\end{eqnarray*}
which completes the proof.
 \end{proof}

 Remark 2. The restriction on $N$ in Theorem~\ref{patternthm} is needed. For example, for the Rudin-Shapiro sequence we have
$$M({\cal P}_2,N)=\left\{\begin{array}{ll} 0, & 1\le N\le 3,\\ 3, & 4\le N\le 9,\\6, &10\le N\le 24.\end{array}\right.$$

Remark 3. For $k\ge 2$ and $N\ge 2^{k+3}-7$ Theorem~\ref{patternthm} implies
$$\frac{N}{6}+1\leq \frac{2^{k-1}-1}{2^k-1}\frac{N}{2}+1\le M({\cal P}_k,N)\le \frac{2^{k-1}-1}{2^k-1}(N-1)+1< \frac{N+1}{2}.$$

 \section{Final remarks}

 The subsequence of the Thue-Morse sequence along $(t_{i^2})_{i=0}^\infty$ is not automatic. Hence, its expansion complexity is unbounded. It is shown by the authors in \cite{suwi}
 that its $N$th maximum order complexity is at least of order of magnitude $N^{1/2}$ and this sequence may be an attractive candidate for cryptographic applications. Pattern sequences along squares are also
 analyzed in \cite{suwi}.

 The correlation measure of order $k$ introduced by Mauduit and
 S\'ark\"ozy \cite{masa97}
  is another figure of merit which is finer than the linear complexity,
 see \cite{brwi06}.
  A cryptographic sequence must have small correlation measure of all orders~$k$ up to a sufficiently large $k$.  In \cite{iswi17}
   the maximum order complexity of a binary sequence was estimated
  in terms of its correlation measures. Roughly speaking, it was shown that
  any sequence with small correlation measure up to a sufficiently large
  order $k$ cannot have very small maximum order complexity. Moreover, the
 correlation measure of order $2$ of both Thue-Morse and Rudin-Shapiro
 sequence of length $N$ is of order of magnitude $N$, see \cite{masa98}.
  The same is true for any pattern sequence, see \cite{mewi17}.
  Hence, together with the results of this paper we see that the correlation
 measure of order $k$ is a finer quality measure for cryptographic sequences
 than the maximum order complexity.

 Combining a bound of \cite{br17}
  on the state complexity in terms of the expansion complexity and a bound
 of \cite{mewi17}
 on the state complexity in terms of the correlation measure of order $2$, we
 can also estimate the expansion complexity in terms of the correlation
 measure of order $2$.

 Furthermore, the maximum order complexity and its connections with
 Lempel-Ziv complexity was studied in \cite{likoka07}.

  In \cite{suzelihe17}
  the (periodic) sequences of largest possible maximum order complexity were
 classified. However, these sequences are highly predictable and not
 suitable in cryptography.  In  \cite{luxiyo17}
 and \cite{nixi14}
  several sequence constructions are given which have very large maximum order
 complexity but no obvious flaw.

 Finally, we mention that although the linear complexity is a weaker quality
 measure for cryptographic sequences than maximum order complexity as well
 as correlation measure and expansion complexity, it is
 still of high practical importance since it is much easier to calculate than
 all of the finer measures.

\section*{Acknowledgments}
The first author is  supported by China Scholarship Council and the National Natural Science Foundation of China Grant 61472120 .
The second author is partially supported by the Austrian Science Fund FWF Project P~30405-N32.


\begin{thebibliography}{99}

 \bibitem{brwi06} N. Brandst\"atter, A. Winterhof, Linear complexity profile of binary
 sequences with small correlation measure. Period. Math. Hungar. 52 (2006),
 1--8.

  \bibitem{br17} A. Bridy, Automatic sequences and curves over finite fields. Algebra Number
 Theory 11 (2017), 685--712.

 \bibitem{di12} C.\ Diem, On the use of expansion series for stream ciphers. LMS J. Comput. Math. 15 (2012), 326--340.

 \bibitem{ermu} D. Erdmann, S. Murphy, An approximate distribution for the
maximum order complexity. Des. Codes Cryptogr. 10 (1997), 325--339.

 \bibitem{gu} F. G. Gustavson, Analysis of the Berlekamp-Massey linear feedback shift-register
synthesis algorithm. IBM J. Res. Develop. 20 (1976), 204--212.

\bibitem{iswi17}  L. I\c s\i k, A. Winterhof,  Maximum-order complexity and correlation
 measures. Cryptography 1 (2017), 7, 1--5.

 \bibitem{ja89} C. J. A. Jansen, Investigations on nonlinear streamcipher systems: construction and evaluation methods. Ph.D.\ dissertation, Technical University of Delft, Delft, 1989.

\bibitem{ja91} C. J. A. Jansen, The maximum order complexity of sequence ensembles.
D.W. Davies (Ed.): Advances in Cryptology - EUROCRYPT '91, Lect. Notes Comput. Sci. 547, pp. 153--159,
Springer-Verlag, Berlin Heidelberg, 1991.

\bibitem{jabo} C. J. A. Jansen, D. E. Boekee, The shortest feedback shift
register that can generate a given sequence. G. Brassard (Ed.): Advances in
Cryptology--CRYPTO. Lect. Notes Comput. Sci. 435, pp. 90--99,
Springer-Verlag, Berlin Heidelberg, 1990.

\bibitem{likoka07} K. Limniotis, N. Kolokotronis, N. Kalouptsidis,
 On the nonlinear complexity and Lempel-Ziv complexity of finite length
 sequences. IEEE Trans. Inform. Theory 53 (2007), 4293--4302.

 \bibitem{luxiyo17} Y. Luo, C. Xing, L. You, Construction of sequences with high nonlinear
 complexity from function fields. IEEE Trans. Inform. Theory 63 (2017), 7646--7650..

 \bibitem{masa97} C. Mauduit, A. S\'ark\"ozy, On finite pseudorandom binary
 sequences I: Measure of pseudorandomness, the Legendre symbol. Acta Arith.
 82 (1997), 365--377.

 \bibitem{masa98} C. Mauduit, A. S\'ark\"ozy,
 On finite pseudorandom binary sequences: II. The Champernowne,
 Rudin-Shapiro, and Thue-Morse sequences, a further construction.
 J. Number Theory 73 (1998), 256--276.

 \bibitem{MW13} W.\ Meidl, A.\ Winterhof, Linear complexity of sequences and multisequences. Handbook of finite fields (G.L.\ Mullen, D.\ Panario, eds.),
324--336, CRC Press, Boca Raton, FL, 2013.

 \bibitem{meniwi} L. M\'erai, H. Niederreiter, A. Winterhof, Expansion complexity and linear complexity of sequences over finite fields. Cryptogr. Commun. 9 (2107), 501--509.

 \bibitem{mewi17} L. M\'erai, A. Winterhof, On the pseudorandomness of automatic sequences. Cryptogr. Commun. 10 (2018), 1013--1022.

 \bibitem{mewi} L. M\'erai, A. Winterhof, On the $N$th linear complexity of automatic sequences.  J. Number Theory 187 (2018), 415--429.

\bibitem{nixi14} H. Niederreiter, C. Xing, Sequences with high nonlinear complexity. IEEE Trans. Inform. Theory 60 (2014),  6696--6701.

\bibitem{suwi} Z. Sun, A. Winterhof, On the maximum order complexity of subsequences of the Thue-Morse and Rudin-Shapiro sequence along squares. Int. J. Comput. Math. Comput. Syst. Theory 4 (2019), 30--36.

\bibitem{suzelihe17} Z. Sun, X. Zeng, C. Li, T. Helleseth, Investigations on periodic sequences
 with maximum nonlinear complexity. IEEE Trans. Inform. Theory 63 (2017), 6188--6198.

 \end{thebibliography}
\end{document}